\documentclass[10pt,a4paper]{article} 

\usepackage[latin1]{inputenc}  
\usepackage[babel=true]{csquotes}
\usepackage{amsmath,amscd,amssymb,latexsym,graphicx,color,mathrsfs} 
 \usepackage{hyperref}
 \usepackage[all]{xy}
\newtheorem{theorem}{Theorem}[section]

\newtheorem{proposition}[theorem]{Proposition}
\newtheorem{corollary}[theorem]{Corollary}

\newtheorem{definition}[theorem]{Definition}

\numberwithin{equation}{section} 
 
\newcommand{\cqfd}{\hfill{\small $\Box$}} 
 \newenvironment{proof}[1][]{{\bf Proof #1 : }}{\hfill \cqfd} 
 
\reversemarginpar

\newcommand{\To}{\longrightarrow}

\newcommand{\cn}{\mathrm{cn}} 
\newcommand{\pcn}{\mathrm{pcn}}

\newcommand\tgt[1]{{}^{T}\kern-1pt #1}
\newcommand\adi[1]{{}^{ad}\kern-1pt #1}

 \def\QQ{{\mathbb{Q}}} \def\RR{{\mathbb{R}}}

 \def\ZZ{{\mathbb{Z}}}

 \def\bZ{{\mathbf{Z}}}
\def\cA{{\mathcal{A}}}  
  \def\cF{{\mathcal{F}}}
 \def\cH{{\mathcal{H}}} 
 \def\cK{{\mathcal{K}}} 
  
\def\cP{{\mathcal{P}}}  
\def\cS{{\mathcal{S}}}

\title{On Fredholm boundary conditions on manifolds with corners I: Global corner's cycles obstructions.\footnote{The first two authors were partially supported by ANR-14-CE25-0012-01.} }

\author{P. Carrillo Rouse, J.M. Lescure and M. Velasquez} 
 
\date{\today} 
 
\begin{document}

\maketitle 

\begin{center}
{\bf Abstract}
\end{center}

\noindent Given a connected manifold with corners of any codimension there is a very basic and computable homology theory called conormal homology defined in terms of faces and orientations of their conormal bundles, and whose cycles correspond geometrically to corner's cycles. 

\noindent Our main theorem is that, for any manifold with corners $X$ of any codimension, there is a natural and explicit morphism
$$K_*(\cK_b(X)) \stackrel{T}{\longrightarrow} H^{pcn}_*(X,\mathbb{Q})$$
between the $K-$theory group of the algebra $\cK_b(X)$ of $b$-compact operators for $X$ and  the periodic conormal homology group with rational coeficients, and that $T$ is a rational isomorphism. 

\noindent As shown by the first two authors in a previous paper this computation implies that the rational groups $H^{pcn}_{ev}(X,\mathbb{Q})$ provide an obstruction to the Fredholm perturbation property for compact connected manifold with corners.

\noindent The difference with respect to the previous article of the first two authors in which they solve this problem for low codimensions is that we overcome in the present article the problem of computing the higher spectral sequence K-theory differentials associated to the canonical filtration by codimension by introducing an explicit topological space whose singular cohomology is canonically isomorphic to the conormal homology and whose K-theory is naturally isomorphic to the $K-$theory groups of the algebra $\cK_b(X)$.

\tableofcontents

\section{Introduction}


In this paper we continue our study on obstructions on Fredholm boundary conditions on manifolds with corners initiated in \cite{CarLes}, let us explain and motivate the context of the present work. On a smooth compact manifold, ellipticity of (classical) pseudodifferential operators is equivalent to Fredholmness, and the vanishing of the Fredholm index of an elliptic  pseudodifferential operator is equivalent to its invertibility after perturbation by a regularizing operator. In the case of a smooth manifold with boundary, not every elliptic operator is Fredholm and it is known since Atiyah and Bott that there exist obstructions to  the existence of local boundary conditions in order to upgrade an elliptic operator into a Fredholm boundary value problem. Nonetheless, if one moves to non-local boundary conditions, obstructions disappear: for instance, not every elliptic (b-operator)   pseudodifferential operator is Fredholm but it can be perturbed with a regularizing operator to become Fredholm. This non trivial fact, which goes back to Atiyah, Patodi and Singer \cite{APS1}, can also be obtained from the vanishing of a boundary analytic index 
(see \cite{Mel-Pia1997-1,Mel-Pia1997-2,MontNis}, and below). In fact, in this case the boundary analytic index takes values in the $K_0$-theory group of the algebra of regularizing operators and this K-theory group is easily  seen to vanish. It is known that obstructions to the existence of perturbations of elliptic operators into Fredholm ones reappear in the case of manifolds with corners of arbitrary codimension (\cite{Bunke,NSS2010}) (this includes for instance many useful domains in Euclidean spaces). In this paper we will show that the global topology/geometry of the corners and the way the corners form cycles enter in a fundamental way in a primary obstruction to give Fredholm boundary  conditions. As we will see the answer passes by the computation of some $K$-theory groups. We explain now with more details the problem and the content of this paper.

Using K-theoretical tools for solving index problems was the main asset in the series of papers by Atiyah-Singer (\cite{AS,AS3}).  In the case of manifolds with boundary, $K$-theory is still useful to understand the vanishing of the obstruction to the existence of pertubations of elliptic operators into Fredholm ones (even if $K$-theory is not essential in the computation of analytical indices \cite{APS1}), and a fortiori to understand this obstruction in the case of families of manifolds with boundary (\cite{Mel-Pia1997-1,Mel-Pia1997-2,Mel-Roc2006}).  
For manifolds with corners, Bunke \cite{Bunke} has delivered  for Dirac type operators a complete study of the obstruction, which lives in the homology of a complex associated with the faces of the manifold. As  observed in $\cite{CarLes}$, this homology also appears as the $E^2$-term of the spectral sequence computing the $K$-group that contains the obstruction to Fredholmness for general elliptic $b$-pseudodifferential operators. Nazaikinskii, Savin and Sternin also use $K$-theory to express the obstruction for their pseudodifferential calculus on manifolds with corners and stratified spaces \cite{NSS2,NSS2010}.

Let us briefly recall the framework in which we are going to work.  The algebra of pseudodifferential operators $\Psi_b^*(X)$ associated to any manifold with corners $X$ is defined in \cite{MelPia}: it generalizes the case of manifolds with boundary treated in \cite{Mel} (see also \cite[Section 18.3]{Horm-3}).  The elements in this algebra are called $b-$pseudodifferential operators\footnote{To simplify we discuss only the case of scalar operators, the passage to operators acting on sections of vector bundles is done in the classic way.}, the subscript $b$ identifies these operators as obtained by microlocalization of the Lie algebra of $C^\infty$ vector fields on $X$ tangent to the boundary. This Lie algebra of vector fields can be explicitly obtained as sections of the so called $b$-tangent bundle $^bTX$ (compressed tangent bundle that we will recall below).  The b-pseudodifferential calculus has the classic and expected properties. In particular there is a principal symbol map
$$\sigma_b:\Psi_b^m(X)\to S^{[m]}(^bT^*X).$$  
Ellipticity has the usual meaning, namely invertibility of the principal symbol. Moreover (discussion below and Theorem 2.15 in \cite{MelPia}), an operator is elliptic if and only\footnote{Notice that this remark implies that to an elliptic $b$-pseudodifferential operator one can associate an "index" in the algebraic K-theory group $K_0(\Psi_b^{-\infty}(X))$ (classic construction of quasi-inverses).} if it has a quasi-inverse modulo $\Psi_b^{-\infty}(X)$. Now,  $\Psi_b^{-\infty}(X)$ also contains non compact operators and compacity is there characterized by the vanishing of a suitable indicial map (p.8 ref.cit.). Elliptic $b$-pseudodifferential operators being invertible modulo compact operators -and hence Fredholm\footnote{see p.8 in \cite{MelPia} for a characterization of Fredholm operators in terms of an indicial map or \cite{Loya} thm 2.3 for the proof of Fully ellipticity iff Fredholm}-, are usually said to be  {\sl fully} elliptic. 

The norm  closure $\cK_b(X)$  of  $\Psi_b^{-\infty}(X)$ into the bounded operators on $L^2_b(X)$ fits in the short exact sequence of $C^*$-algebras:
\begin{equation}\label{Introbcompact}
\xymatrix{
0\ar[r]&\cK(X)\ar[r]^-{i_0}&\cK_b(X)\ar[r]^-{r}&\cK_b(\partial X)\ar[r]&0
}
\end{equation}
It will be fundamental in this paper to have a groupoid description of the involved $C^{*}$-algebras. This is explained in Section \ref{sectionAnalyticindexmorphism}. 

In order to understand how the above sequence enters into the study of Fredholm Perturbation properties we need to settle some definitions.

{\bf Analytic and Boundary analytic Index morphism:}  Every lliptic $b$-pseudodifferential operator has  a principal symbol class in $K^0_{top}(^bT^*X)$, and possesses an  interior parametrix that brings a class in $K_0(\cK_b(X))$ called the {\sl analytical index class}. Both classes are related as follows. Consider the short exact sequence 
\begin{equation}\label{IntrobKses}
\xymatrix{
0\ar[r]&\cK_b(X)\ar[r]&\overline{\Psi_b^0(X)}\ar[r]^-{\sigma_b}&C(^bS^*X)\ar[r]&0.
}
\end{equation}
After applying the $K$-functor, it gives rise to the boundary morphism:
$K_1(C(^bS^*X))\to K_0(\cK_b(X))$ that can be factorized canonically into a  morphism
\begin{equation}\label{intro:ana-ind}
\xymatrix{
K^0_{top}(^bT^*X)\ar[r]^-{Ind^a_X}&K_0(\cK_b(X))
}
\end{equation}
called {\it the Analytic Index morphism of $X$}, which is the one that maps the principal symbol class to the analytical index class of a given elliptic $b$-operator. 

Alternatively, we can compose \eqref{intro:ana-ind} by $r$ in \eqref{IntrobKses}:
\begin{equation}\label{Banalyticindex}
\xymatrix{
K^0_{top}(^bT^*X)\ar[r]^-{Ind^\partial_X}&K_0(\cK_b(\partial X))
}
\end{equation}
and call the result {\it the Boundary analytic index morphism of $X$}. In fact $r:K_0(\cK_b(X))\to K_0(\cK_b(\partial X))$ is an isomorphism if $\partial X\not=\emptyset$, proposition 5.6 in \cite{CarLes}, and so the two indices above are essentially the same. This describes the role of \eqref{Introbcompact} into the relationship between ellipticity and $K$-theory. Notice that in particular there is no contribution of the Fredholm index in the $K_0$-analytic index. 

To express how \eqref{Introbcompact} and the previous index maps relate to the Fredholm obstruction we introduce the following vocabulary:
\begin{definition} Let $D\in \Psi_b^m(X)$ be elliptic. We say that $D$ satisfies: 
\begin{itemize}
\item  the {\it Fredholm Perturbation Property} $(\cF\cP)$ if there is   $R\in \Psi_b^{-\infty}(X)$ such that $D+R$ is  fully elliptic.  
\item the  {\it stably  Fredholm Perturbation Property} $(\cS\cF\cP)$ if $D\oplus 1_H$ satisfies $(\cF\cP)$ for some identity operator $1_H$.
\end{itemize}
\end{definition}
 The following result is due to  \cite{NSS2} (see \cite{CLM} for an alternative proof using deformation groupoids). 
\begin{theorem}\label{intro:recall-thm}
Let $D$ be an elliptic $b$-pseudodifferential operator on a compact manifold with corners $X$. Then  $D$ satisfies $(\cS\cF\cP)$ if and only if 
\(
 Ind^\partial_X([\sigma_b(D)])=0 \,\,\text{in}\,\,  K_0(\cK_b(\partial X))
\).\\
In particular if $D$ satisfies $(\cF\cP)$ then its boundary analytic index vanishes. 
\end{theorem}
The present work  is motivated by the understanting of the $K$-groups that capture this obstruction, preferably in terms of the geometry and topology of the manifold with corners. As it happens, the only previously known cases are: 
\begin{itemize}
\item the $K$-theory of the compact operators $\cK(X)$, giving $K_0(\cK(X))=\bZ$ and $K_1(\cK(X))=0$, which is of course essential for classic index theory purposes; 
\item the $K$-theory of $\cK_b(X)$ for a smooth manifold with boundary, giving $K_0(\cK_b(X))=0$ and $K_1(\cK_b(X))=\bZ^{p-1}$ with $p$ the number of boundary components, which has the non trivial consequence that any elliptic $b$-operator on a manifold with boundary can be endowed with Fredholm boundary conditions;
\item the $K$-theory of $\cK_b(X)$ for $X$ a finite product of manifolds with corners of codimension less or equal to three, where we recall below the computation given in \cite{CarLes} in terms of very computable homology associated to the corner's faces. In particular, as shown by examples in ref.cit., from codimension 2 any possible free abelian group can arise as one of this K-theory groups. 
\end{itemize}

In this paper we address the case of arbitrary manifolds with corners and the geometrical tool used to describe the needed $K$-theory groups is {\sl conormal homology} \cite{Bunke,CarLes}. The conormal complex, initially considered in \cite{Bunke}, is made of faces provided with orientation of their conormal (trivial) bundle and the differential maps a co-orientated face of codimension $p$ to the sum of codimension $p-1$ faces -provided with the induced co-orientation-,   that contains it in their closures.   Bunke proved that the obstruction for the existence of a {\sl boundary taming} of a Dirac type operator on a manifold with corners $X$ is given by an explicit class in this homology (which also implicitly appears in the work of Melrose and Nistor in \cite{MelNis}, through the quasi-isomorphism in \cite[Corollary 5.5]{CarLes}. It is thus all but a surprise that conormal homology emerges from the computation of $K_*(\cK_b(X))$. Just recall that the faces decomposition of $X$ yields a filtration 
\begin{equation}\label{eq:K-filt0}
   \cK(L^2(\overset{\circ}{X}))=A_0\subset A_1 \subset \ldots\ A_d=A=\cK_b(X)
\end{equation}
already used by Melrose and Nistor in \cite{MelNis} and one of their main result is the expression of  the first differential (theorem 9 ref.cit.) of the corresponding spectral sequence $(E^*_{*,*}(\cK_b(X)),d^*_{*,*})$. In \cite{CarLes}, it is proved that:
\begin{equation}
   H_{p}^{\cn}(X) \simeq E^2_{p,0}(\cK_b(X)),
\end{equation}
and furthermore that the even/odd groups $H_{ev/odd}(X,\QQ)$ are isomorphic with $K_{0/1}(\cK_b(X))\otimes\QQ$  for product of manifolds of codimension at most 3 (tensorisation by $\QQ$ can be dropped of if $X$ itself is of codimension at most $3$ or if one factor is of codimension at most 2). This gave a concrete geometric description to the computations initiated in \cite{MelNis}, but after the paper \cite{CarLes}, the expression of the higher differentials remained unclear in the case of arbitrary codimension, that is, it remained unclear whether conormal homology is a satisfactory geometric replacement for $K_0(\cK_b(X))$. 

In this paper we overcome this problem by  using groupoid methods in order to find a topological space  $O_X$ equivalent to $\cK_b(X)=C^{*}(\Gamma_{b}(X))$ in $K$-theory (Section  \ref{secOX}). This leads to the main result of this paper, which generalizes \cite[ theorem 5.8]{CarLes}.

\begin{theorem}\label{intro:main-thm}
For every connected manifold with corners $X$ there are morphisms
\begin{equation}\label{intro:main-map}
T_{ev/odd}:K_{ev/odd}(\cK_b(X))\longrightarrow  H_{ev/odd}^{cn}(X)\otimes\mathbb{Q}
\end{equation}
inducing rational isomorphisms.
Explicitly, $T_*$ is given by the composition of
\begin{enumerate}
\item The Connes-Thom isomorphism 
\begin{equation}
CT_h:K_*(\cK_b(X))\to K^*_{top}(O_X),
\end{equation}
\item the Chern character morphism
\begin{equation}
K^*_{top}(O_X)\stackrel{ch}{\longrightarrow}H^{ev/odd}(O_X)\otimes\mathbb{Q},
\end{equation}
which is a rational isomorphism
and
\item the natural isomorphism 
\begin{equation}
B_*\otimes Id:H^{*}(O_X)\otimes\mathbb{Q}\stackrel{}{\longrightarrow}H_{*}^{pcn}(X)\otimes\mathbb{Q},
\end{equation}
described in the section \ref{subsechomOXvscnhom}.
\end{enumerate}
\end{theorem}
We do not expect the isomorphism above to hold without tensorizing by $\QQ$ in general, even if many examples exist: manifolds of codimension at most $3$, product of such manifolds with one factor of codimension at most $2$, manifolds whose even conormal homology is torsion free.


Nevertheless, it is a straightforward consequence of Theorems \ref{intro:recall-thm} and \ref{intro:main-thm} that the rational groups $ H_{ev}^{cn}(X,\QQ)$ provide an obstruction to the Fredholm perturbation properties.  As observed   in \cite{CarLes},  the groups $H_{ev}^{cn}(X)$ tend to be non trivial. Even more interesting, for naturally geometric operators the boundary analytic indices do not vanish neither. Hence this motivates the study of the composition of \eqref {intro:ana-ind} with \ref{intro:main-map}, which will be called  {\sl Corner index morphism}:
\begin{equation}
K^0_{top}(^bT^*X)\stackrel{Ind^X_{cn}}{\longrightarrow} H_{ev}^{pcn}(X)\otimes \mathbb{Q}.
\end{equation}
This brings in natural questions such as:
\begin{enumerate}
\item Given an elliptic operator $D\in \Psi^*_b(X)$, can we express, in terms of corner cycles, the class
$Ind^X_{cn}([\sigma_D])$ ? And is it possible to refine the computation at the integral level ? 
\item Is it possible to give a topological  formula (in the spirit of the  Atiyah-Singer theorem) for  the Corner index morphism and then for the obstructions for Fredholm perturbations ? 
\end{enumerate}

We will study these questions in a subsequent paper.

\vspace{2mm}

{\bf Acknowledgements.} The first two authors want to thank Victor Nistor for very helpful discussions and for suggesting us a computation using the classic Chern character. The first and third authors want to thank the ANR SINGSTAR for support two collaboration research visits in Toulouse.

\section{Index theory for manifolds with corners via groupoids}\label{sectionAnalyticindexmorphism}

\subsection{Groupoids and manifolds with corners}

For background and notation about Lie groupoids and their relationship to $C^*$-algebras, $K$-theory, index theory and pseudodifferential analysis, the reader may consult \cite{DL10,NWX,MP,LautNist,LMV,HS83,Ren,AnRen} and references therein. For backgound about $b$-calculus the reader may consult \cite{Mel,Melmwec,MelPia,Loya} and references therein.
Here we closely follows the definition and notation of \cite{CarLes,CLM}.

We will consider  compact manifolds $X$ with embedded corners \cite{Melmwec}: we thus may fix once for all  a smooth compact manifold  $\tilde{X}$ 
and submersions $\rho_1,..., \rho_n : \tilde{X}\longrightarrow \RR$  such that:
\begin{enumerate}
\item  $ X = \bigcap_{1\le j\le n} \rho_j^{-1}([0,+\infty)) \subset \tilde{X}$ 
\item Setting $H_j = \rho_j^{-1}( \{0\})\cap X$, $j=1,\ldots,n$, we require that $\{ d\rho_{j_{1}},...,d\rho_{j_{k}}\}$ has maximal rank at any point of $ H_{j_{1}}\cap\ldots\cap H_{j_{k}}$ for any $1\le j_{1} < \cdots < j_{k} \le n$. 
\end{enumerate}
We assume for simplicity that all the boundary hypersurfaces $H_j$ of $X$ are connected, as well as $X$ itself. 

The so-called Puff groupoid  \cite{Mont} is then defined by:
\begin{equation}\label{Puffgrpd}
G(\tilde{X},(\rho_i))= \{(x,y,\lambda_1,...,\lambda_n)\in \tilde{X}\times \tilde{X}\times \RR^n: \rho_i(x)=e^{\lambda_i}\rho_i(y)\}.
\end{equation}
This is a Lie subgroupoid of $\tilde{X}\times \tilde{X}\times \RR^k$. 
The $b$-groupoid $\Gamma_{b}(X)$  \cite{Mont} is then defined as the $s$-connected component of the restriction of the Puff groupoid to (the saturared closed subspace) $X$.  It is again a Lie amenable groupoid (in the extended sense of \cite{LautNist}) whose Lie algebroid identifies in a canonical way with the compressed tangent bundle ${}^bTX$. Also, the vector representation (or the regular one at any interior point of $X$) of the algrebra of compactly supported pseudodifferential operators on $\Gamma_{b}(X)$  is equal to the compactly supported small $b$-calculus. This equality can be enlarged to the small calculus by  adding a natural Schwartz space of $\Gamma_{b}(X)$ \cite{LMN} to the pseudodifferential $\Gamma_{b}(X)$-calculus but the operation is not necessary for $K$-theory purposes. Indeed, denoting by  $\cK_b(X)$ the closure of $\Psi^{-\infty}_b(X)$ into the algebra of bounded operators on $L^2_b(X)$, we have a natural isomorphism
\begin{equation}
C^*(\Gamma_b(X))\cong \cK_b(X).
\end{equation}

We will now introduce the several index morphisms we will be using, mainly the Analytic and the Fredholm index.
In all this section, $X$ denotes a compact and connected manifold with embedded corners. 

\subsection{Ellipticity and Analytical Index morphisms}\label{subsectionAnalyticindexmorphism}

The analytical index morphism (of the manifold with embedded corners $X$)  takes it values in the group $K_0(\cK_b(X))$. It can be defined in two ways. First, we may consider  the connecting homomorphism $I$ of the exact sequence in $K$-theory associated with the short exact sequence of $C^*$-algebras:
\begin{equation}\label{bKses}
\xymatrix{
0\ar[r]&\cK_b(X)\ar[r]&\overline{\Psi_b^0(X)}\ar[r]^-{\sigma_b}&C(^bS^*X)\ar[r]&0.
}
\end{equation}
Then, if $[\sigma_b(D)]_1$ denotes the class in $K_1(C({}^bS^*X))$ of the principal symbol $\sigma_b(D)$ of an elliptic $b$-pseudodifferential $D$, we define the    analytical index   $\mathrm{Ind}_{\mathrm{an}}(D)$ of $D$ by 
\[
\mathrm{Ind}_{\mathrm{an}}(D)=I([\sigma_b(D)]_1)\in K_0(\cK_b(X)).
\]
Secondly, we can in a first step produce a $K_0$-class $[\sigma_b(D)]$ out of $\sigma_b(D)$: 
\begin{equation}
 [\sigma_b(D)] = \delta([\sigma_b(D)]_1)\in K_0(C_0({}^bT^*X))
\end{equation}
where $\delta$ is the connecting homomorphism of the exact sequence relating the vector and  sphere bundles:
\begin{equation}\label{bTS}
\xymatrix{
0\ar[r]& C_0({}^bT^*X)\ar[r]& C_0({}^bB^*X)\ar[r]&C({}^bS^*X)\ar[r]&0.
}
\end{equation}
Next,  we consider the exact sequence coming with the adiabatic deformation of $\Gamma_{b}(X)$:
\begin{equation}
\xymatrix{
0\ar[r]&C^*(\Gamma_b(X)\times (0,1])\ar[r]&C^*(\Gamma_b^{tan}(X))\ar[r]^-{r_0}&C^*(^bTX)\ar[r]&0,
}
\end{equation}
 in which the ideal is $K$-contractible. Using the shorthand notation $K^0_{top}(^bT^*X)$ for $K_0(C^*(^bTX))$, we set: 
\begin{equation}
Ind^a_X= r_1 \circ r_0^{-1} : K^0_{top}(^bT^*X)\longrightarrow K_0(\cK_b(X))
\end{equation}
where $r_1 : K_0(C^*(\Gamma_b^{tan}(X)))\to K_0(C^*(\Gamma_b(X))) $ is induced by the restriction morphism to $t=1$.
Applying a mapping cone argument to the exact sequence (\ref{bKses}) gives a commutative diagram
\begin{equation}
\xymatrix{
K_1(C(^bS^*X))\ar[rd]_-{\delta}\ar[rr]^-{I}&&K_0(\cK_b(X))\\
&K^0_{top}(^bT^*X)\ar[ru]_-{Ind^a_X}&
}
\end{equation}
Therefore we get, as announced:
\begin{equation}
  \mathrm{Ind}_{\mathrm{an}}(D) = Ind^a_X([\sigma_b(D)])
\end{equation}
 The map $Ind^a_X$ will be called  the {\it  Analytic Index morphism} of $X$.  A closely related homomorphism is the {\it  Boundary analytic Index morphism}, in which  the restriction to $X\times\{1\}$ is replaced  by the one to $\partial X\times\{1\}$, that is, we set:
 \begin{equation}
  Ind^\partial_X =  r_\partial \circ r_0^{-1}  : K_0(C_0(^bT^*X))\longrightarrow K_0(C^*(\Gamma_b(X)|_{\partial X})),
 \end{equation}
 where $r_\partial$ is induced by the   homomorphism  
 $C^*(\Gamma^{tan}_b(X))\longrightarrow C^*(\Gamma_b(X))|_{\partial X} $. We have of course 
 \begin{equation}
  Ind^\partial_X = r_{1,\partial}\circ Ind^a_X 
 \end{equation}
if $r_{1,\partial}$ denotes the map induced by the homomorphism  $C^*(\Gamma_b(X))\longrightarrow C^*(\Gamma_b(X)|_{\partial X})$. Since $r_{1,\partial}$ induces an isomorphism between $K_0$ groups (proposition 5.6 in \cite{CarLes}), both indices have the same meaning.

\subsection{Full ellipticity and the Fredholm Index morphism}\label{Fredsubsection}
To capture  the defect of Fredholmness of elliptic $b$-operators on $X$, we may introduce the algebra of full, or joint, symbols $\cA_\cF$ \cite{LMNpdo}.  If $F_1$ dnoytes the set of closed boundary hypersurfaces of $X$,  then the full symbol map is the $*$-homomorphism given by: 
\begin{equation}
 \sigma_F : \Psi^0(\Gamma_b(X))\ni P \longmapsto \Big( \sigma_b(P),(P\vert_H)_{H\in F_1}\Big) \in \cA_{\cF}.
\end{equation}
It gives rise to the exact sequence: 
\begin{equation}\label{Fredses}
\xymatrix{
0\ar[r]&\cK(X)\ar[r]&\overline{\Psi^0(\Gamma_b(X))}\ar[r]^-{\sigma_F}& \overline{\cA_{\cF}}\ar[r]&0
}
\end{equation}
where $\cK(X)$ is the algebra of compact operators on $L^2_b(X)$.
An operator $D\in \Psi^0(\Gamma_b(X))$ is said to be fully elliptic if $\sigma_F(D)$ is invertible.
 In \cite{Loya} (the statement also appears in \cite{MelPia}), it is proved that full ellipticity is equivalent to Fredholmness  on any $b$-Sobolev spaces $H^s_b(X)$.

For a given fully elliptic operator $D$, we denote by $\mathrm{Ind}_{\mathrm{Fred}}(D)$ its Fredholm index. We briefly recall how this integer is captured in $K$-theory. First, there is a natural isomorphism 
\begin{equation}
     K_0(\mu) \cong K_0(C^*(\mathcal{T}_{nc}X))
\end{equation}
between the $K$-theory of the obvious homomorphism $C(X)\To \overline{\cA_{\cF}}$ and the $K$-theory of the noncommutative tangent space $\mathcal{T}_{nc}X$.  The former $K$-group captures stable homotopy classes of fully elliptic operators and the latter, which comes from deformation groupoid techniques,   classifies the noncommutative symbols $\sigma_{nc}(D)$ of  fully elliptic operators $D$.  

    Next, the same deformation  techniques give rise to a homomorphism:
 \begin{equation}\label{Fredmorph}
Ind^X_F : K^0(T_{nc}X)\longrightarrow
K_0(\cK(X))\simeq \mathbb{Z},
\end{equation}
which satisfies:
\begin{equation}
Ind_F^X([\sigma_{nc}(D)])= \mathrm{Ind}_{\mathrm{Fred}}(D),
\end{equation}
for any fully elliptic operator $D$. 

\subsection{Obstruction to full ellipticity and Fredholm perturbation property }\label{Obstructionsection}
 
 In order to analyse the obstruction to full ellipticity, we introduce Fredholm Perturbation Properties \cite{NisGauge}. 
\begin{definition} Let $D\in \Psi_b^m(X)$ be elliptic. We say that $D$ satisfies: 
\begin{itemize}
\item  the {\it Fredholm Perturbation Property} $(\cF\cP)$ if there is   $R\in \Psi_b^{-\infty}(X)$ such that $D+R$ is  fully elliptic. 
\item the  {\it stably  Fredholm Perturbation Property} $(\cS\cF\cP)$ if $D\oplus 1_H$ satisfies $(\cF\cP)$ for some identity operator $1_H$.
\item the {\it stably homotopic Fredholm Perturbation Property} $(\cH\cF\cP)$ if there is a fully elliptic operator $D'$ with $[\sigma_b(D')]=[\sigma_b(D)]\in K_0(C^*({}^bTX))$.
\end{itemize}
\end{definition}
We also say that $X$ satisfies the   {\it (resp. stably) Fredholm Perturbation Property}  if any elliptic $b$-operator on $X$ satisfies $(\cF\cP)$ (resp. $(\cS\cF\cP)$).

Property  $(\cF\cP)$ is  stronger than property $(\cS\cF\cP)$ which in turn is equivalent to property $(\cH\cF\cP)$ by  \cite[Proposition 4.3]{DebSkJGP}.  In  \cite{NSS2}, Nazaikinskii, Savin and Sternin characterized $(\cH\cF\cP)$
 for arbitrary manifolds with corners using an index map associated with their dual manifold construction. In \cite{CarLes} the result of \cite{NSS2} is rephrased in terms of deformation groupoids with the non trivial extra apport of changing $(\cH\cF\cP)$ by $(\cS\cF\cP)$ thanks to   \cite[Proposition 4.3]{DebSkJGP}: 
  
\begin{theorem}\label{AnavsFredthm1}
Let $D$ be an elliptic $b$-pseudodifferential operator on a compact manifold with corners $X$. Then  $D$ satisfies $(\cS\cF\cP)$ if and only if 
$
 Ind_X^\partial([\sigma_b(D)])=0$ in  $K_0(C^*(\Gamma_b(X)|_{\partial X}))$.  \\ In particular, if $D$ satisfies $(\cF\cP)$ then its boundary analytic  index vanishes. 
\end{theorem}
This motivates the computation of $K_0(C^*(\Gamma_b(X)))\cong K_0(C^*(\Gamma_b(X)|_{\partial X}))$.


\section{A topological space $K$-equivalent to  $\Gamma_b(X)$}\label{secOX}

The subject of this section is to construct an explicit space $O_X$ and an explicit Connes-Thom isomorphism 
\begin{equation}
CT:K_*(C^*(\Gamma_b(X)))\longrightarrow K^*_{top}(O_X).
\end{equation}
This will be done by replacing $\Gamma_b(X)$ by an action groupoid which has the same $K$-theory and moreover which is free and proper, and then equivalent to its space of orbits.   The general idea comes from \cite{Concg}, the case of manifolds with boundary is treated in  \cite{CLM} and all the material here comes directly from  \cite{CLM}.

\subsection{The orbit space $O_X$}\label{OXsubsection1}

Consider an embedding 
$$\iota :\widetilde{X}\hookrightarrow  \mathbb{R}^{N-n}$$
with $N$ even and $n$ still denoting the number of boundary hypersurfaces of $X$. Consider the groupoid morphism
\begin{equation}\label{hmorphism}
\xymatrix{
h:\widetilde{X}\times \widetilde{X}\times\mathbb{R}^n\to \mathbb{R}^{N-n}\times \mathbb{R}^n=\mathbb{R}^N
}
\end{equation}
given by 
$$h(x,y,(\lambda_i)_i)=(\iota(x)-\iota(y),(\lambda_i)_i).$$
 
The morphism $h$ induces a semi-direct product groupoid
\begin{equation}
(\widetilde{X}\times \widetilde{X}\times\mathbb{R}^n) \rtimes \mathbb{R}^N\rightrightarrows \widetilde{X}\times \mathbb{R}^N.
\end{equation} 
A very simple and direct computation gives that this groupoid is free (that is, has trivial isotropy subgroups) and proper, the freeness commes from the fact $h$ is a monomorphism of Lie groupoids (if $h(\gamma)$ is a unit then $\gamma$ itself is a unit), and the properness of the map
$$(\widetilde{X}\times \widetilde{X}\times\mathbb{R}^n) \rtimes \mathbb{R}^N\stackrel{(t,s)}{\longrightarrow}(\widetilde{X}\times\mathbb{R}^N)^2$$ 
can be verified by a direct computation.

Now, as shown by Tu in \cite{Tu04} proposition 2.10, a topological groupoid $G\rightrightarrows Z$ is proper iff the asscoiated map $(t,s)$ is closed and the stabilizers are quasi-compact. In particular, since $\Gamma_b(X)$ is a closed subgroupoid of $\widetilde{X}\times \widetilde{X}\times\mathbb{R}^n$ and since the induced groupoid morphism
\begin{equation}
\Gamma_b(X)\stackrel{h}{\longrightarrow} \mathbb{R}^N
\end{equation} 
is a groupoid monomorphism we obtain that: 
 
\begin{proposition}
The semi-direct product groupoid 
\begin{equation}
\Gamma_b(X)\rtimes\mathbb{R}^N \rightrightarrows X\times \mathbb{R}^N
\end{equation}
is free and proper.
\end{proposition}
By  \cite{Tu04} (section 2), the space of orbits $X\times \mathbb{R}^N/\Gamma_b(X)\rtimes\mathbb{R}^N$ is then Hausdorff and locally compact. We let:

\begin{definition}[The orbit space]
We denote by
\begin{equation}
O_X:=Orb(\Gamma_b(X)\rtimes\mathbb{R}^N)
\end{equation}
the Orbit space associated with the groupoid $\Gamma_b(X)\rtimes\mathbb{R}^N \rightrightarrows X\times \mathbb{R}^N$. 
\end{definition}
By  classic groupoid's results recalled for instance in \cite{CLM} section 2, we have the following
\begin{proposition}\label{CTOX}
There is an isomorphism
\begin{equation}
CT_h:K_*(C^*(\Gamma_b(X)))\longrightarrow K^*_{top}(O_X)
\end{equation}
given by the composition of the Connes-Thom isomorphism
$$CT:K_*(C^*(\Gamma_b(X)))\longrightarrow K_*(C^*(\Gamma_b(X)\rtimes\mathbb{R}^N))$$
and the isomorphism 
$$\mu:K_*(C^*(\Gamma_b(X)\rtimes\mathbb{R}^N))\longrightarrow	K^*_{top}(O_X)$$
induced from the groupoid Morita equivalence between $\Gamma_b(X)\rtimes\mathbb{R}^N$ and $O_X$ (seen as a trivial unit groupoid).
\end{proposition}

\subsection{The orbit space $O_X$ as a manifold with corners}

Later on the paper we will need to apply the Chern character morphism to the topological space $O_X$ (to its topological K-theory) and for this we will justify in this section that this space has indeed the homotopy type of a CW-space, in fact this space inherits from $X\times \mathbb{R}^N$ a manifold with corners structure as we will now explain.

Consider the $s$-connected Puff groupoid, recalled in (\ref{Puffgrpd}), 
\begin{equation}
G_c(\tilde{X},(\rho_i))\rightrightarrows \tilde{X}.
\end{equation}
It is a Lie groupoid and the semi-direct product groupoid
\begin{equation}
G_c(\tilde{X},(\rho_i))\rtimes_h\mathbb{R}^N\rightrightarrows \tilde{X}\times\mathbb{R}^N
\end{equation}
induced by the morphism $h$ defined in (\ref{hmorphism}) is a free proper Lie groupoid by exactly the same arguments applied to $\Gamma_b(X)$ in section \ref{OXsubsection1}. By classic results on Lie groupoid theory, the orbit space
\begin{equation}
O_{\tilde{X}}:=Orb(G_c(\tilde{X},(\rho_i))\rtimes_h\mathbb{R}^N)
\end{equation}
inherits from $\tilde{X}\times \mathbb{R}^N$ a structure of a $C^\infty$-manifold. A good reference for this is the nice extended survey of Crainic and Mestre, \cite{CrMes}, that clarifies and explains very interesting results on Lie groupid theory that were confusing in the litterature, in particular they explain the role of the linearization theorem for proper Lie groupoids (theorem 2 in ref.cit.) on the local structure of such groupoids and on their orbit spaces.

We will now give the defining functions on $O_{\tilde{X}}$ whose positive parts will define $O_X$. For this, denote, as in sections above, a vector $v=(v',v'')\in \mathbb{R}^{N-n}\times \mathbb{R}^n$. A simple and direct computation shows that, for $i=1,...n$, the $C^\infty$-map
\begin{equation}
(x,v)\mapsto \rho_i(x)e^{v''_i}
\end{equation}
induces a well defined $C^\infty$-map
\begin{equation}
\tilde{\rho}_i:O_{\tilde{X}}\to \mathbb{R}.
\end{equation}
Using the map (\ref{Qdefinition}) and the induced homeomorphisms on the faces (\ref{eq:ident-f}) one can get that
\begin{equation}
O_X=\bigcap_{i=1,...,n}\{\tilde{\rho}_i\geq 0\}.
\end{equation}

Finally, a simple computation yields
\begin{equation}
d_{(x,A)}R_i(W,V)=e^{A''_i}d_x\rho_i(W)+e^{A''_i}V''_i\rho_i(x)
\end{equation}
where $R_i=\tilde{\rho}_i\circ p$ (with $p:\tilde{X}\times \mathbb{R}^N\to O_{\tilde{X}}$ the quotient map), $(W,V)\in T_x\tilde{X}\times T_A\mathbb{R}^N$, and since $p$ is a submersion we obtain that $\{ d\tilde{\rho}_{j_{1}},...,d\tilde{\rho}_{j_{k}} \}$ has maximal rank at any point of $ \tilde{H}_{j_1}\cap\ldots\cap \tilde{H}_{j_{k}}$ for any $1\le j_{1} < \cdots < j_{k} \le n$, where $H_j = \tilde{\rho}_j^{-1}( \{0\})\cap O_X$.

In conclusion we obtain that $O_X$ is a manifold with embedded corners defined by the defining fuctions $\tilde{\rho}_1,...,\tilde{\rho}_n$, the set of its faces of a given codimension is in bijection with the set of faces of $X$ the corresponding codimension. As proved above, each face is homeomorphic to an euclidien space.
 
Now, it is a classic fact, for example see corollary 1 in \cite{MilCW}, that any topological separable manifold as $O_X$ has the homotopy type of a countable $CW$-complex. This is all we will need in the following sections. 

\subsection{The filtration of $O_X$}

The space $O_X$ is a quotient space $X\times\mathbb{R}^N/\sim$ where the relation is given as follows  $(x,A)\sim (y,B)$ iff there is $\gamma=((x,y),(\lambda_i)_i)\in \Gamma_b(X)$ with $B=h(\gamma)+A$. We denote by $\pi : X\times\mathbb{R}^N \To O_{X}$ the quotient map. This map is open \cite[Prop. 2.11]{Tu04}. 

The space  $X$ is naturally filtrated. Indeed, denote by $F_p$ the set of connected faces of codimension $p$ (and $d$ the codimension of $X$).  For a given face $f\in F_p$,  we define the index set $I(f)$ of $f$ to be the unique tuple  $(i_1,\ldots,i_p)$  such that $1\le i_1 <\ldots < i_p \le n$ and 
\begin{equation}
  f \subset H_{i_1}\cap\ldots\cap H_{i_p}
\end{equation}
where we recall that $H_j = \rho_j^{-1}(\{0\})\subset X$.  The filtration of $X$ is then given by:
\begin{equation}\label{filtration-by-codimension}
  X_{j} = \bigcup_{\stackrel{f\in F}{ \  d-j \le \mathrm{codim}(f)\le d}} f
\end{equation}
Then: 
\begin{equation}
 F_d = X_{0} \subset X_{1} \subset \cdots \subset X_d = X.
\end{equation}
and setting $Y_{p}=\pi(X_{p}\times \RR^{N})$, we get a filtration of $O_X$:
\begin{equation}\label{filtrationOX}
Y_0\subset	Y_1\subset \cdots \subset Y_{d-1}\subset Y_d=O_X
\end{equation}
For any index set $I$ we let:  
\begin{equation}
\mathbb{R}_I^{N}:=\{(y,x)\in \RR^{N-n}\times \RR_{+}^n \ ; \   x_i=0 \text{ if  } i\in I \text{ and  }x_i > 0 \text{ otherwise }\} \subset \RR^{N}
\end{equation} 
and we write $\mathbb{R}_f^{N}$ instead of $\mathbb{R}_{I(f)}^{N}$. We are going to define a  map 
$$Q : X \times \RR^{N}\To \RR^{N-n}\times \RR_{+}^n,$$
 smooth and compatible with the equivalence relation on $X \times \RR^{N}$, whose quotient map $q : O_{X}\To  \RR^{N-n}\times \RR_{+}^n$ induces homeomorphisms: 
\begin{equation}\label{eq:ident-f}
  \pi(f\times \RR^{N})   \simeq \RR_{f}^{N}
\end{equation}
for any face $f$ and 
\begin{equation}\label{eq:ident-f-g}
 \pi((f\cup g)\times \RR^{N})  \simeq \RR_{(f,g)}^{N} := \RR_{f}^{N}\ \cup \ \RR_{g}^{N}
\end{equation}
for any pair $(f,g)\in F_{p}\times F_{p-1}$ such that $f\subset \overline{g}$. 

For that purpose, we write for convenience $e^{A}$ for $(e^{A_{1}},\ldots, e^{A_{k}})$ and $\rho.v$ for $(\rho_{1}v_{1},\ldots,\rho_{k}v_{k})$ for all $k$ and $A,\rho,v\in\RR^{k}$. Also, we use the notation $v =(v',v'')\in R^{N-n}\times\RR^{n}$ for any $v\in\RR^{N} $. We then define 
\begin{equation}\label{Qdefinition}
  x\in X, \ v\in \RR^{N},\quad  Q(x,v) = (\iota(x)+v', \rho(x).e^{v''}).
\end{equation}
It is easy to check that $Q : X\times \RR^{N}\To \RR^{N-n}\times \RR_{+}^n$ is a surjective submersion, compatible with the equivalence relation. We denote by $q : O_{X}\To \RR^{N-n}\times \RR_{+}^n$ the quotient map.
For any $f\in F_{*}$, one can check that 
\begin{equation}
  f\times \RR^{N} = Q^{-1}(\RR_{f}^{N}) \text{ and } \forall x,y\in f, v,w\in\RR^{N},\  Q(x,v)=Q(y,w) \iff (x,v)\sim (y,w). 
\end{equation}
It follows that $q_{|_{f}}$ and $q_{|_{f\cup g}}$ provide the homeomorphisms \eqref{eq:ident-f} and \eqref{eq:ident-f-g}. We have proved: 
\begin{proposition}\label{OXTop2}
For any pair $(f,g)\in F_{d-q}\times F_{d-q-1}$ such that $f\subset \overline{g}$, we have a commutative diagram:
\begin{equation}
\xymatrix{
(Y_q\setminus Y_{q-1})_f\ar[d]^-{\approx}_-{q_{|_{f}}}\ar[r]&(Y_{q+1}\setminus Y_{q-1})_{f\cup g}\ar[d]_-{\approx}^-{q_{|_{f\cup g}}}\\
\mathbb{R}_f^{N}\ar[r]&\mathbb{R}_{(f,g)}^{N}
}
\end{equation}
where the vertical maps are homeomorphisms, the horizontal maps are the inclusions and:
\begin{equation}
(Y_q\setminus Y_{q-1})_f := \pi(f\times \mathbb{R}^N)\quad ; \quad (Y_{q+1}\setminus Y_{q-1})_{f\cup g}:=q((f\cup g)\times \mathbb{R}^N)
\end{equation}
\end{proposition}
This will be used to compute the singular cohomology groups of $O_X$, which requires the understanding of  the inclusions $Y_q\setminus Y_{q-1}\hookrightarrow Y_{q+1}\setminus Y_{q-1}$ and more specifically how they  look like around a given face $f\in F_{d-q}$ with respect to a given $g\in F_{d-q-1}$ with $f\subset \overline{g}$.

\subsection{Cohomology of the orbit space}
We can expect the same difficulties  in computing  $K^*(O_X)$ as the ones   encountered in computing $K_*(C^*(\Gamma_b(X))$ (using the  spectral sequence associated with the corresponding filtrations). Instead,  the spectral sequence argument becomes  simpler for  the (singular) cohomology of $O_X$ with compact support. In this section we follow the notations used in \cite[Sec. 2.2]{mccleary}

The spectral sequence associated to the filtration considered in the last subsection will allow to give a very explicit cohomological computation because of proposition \ref{OXTop2}.

Explicitly we associate a cohomological  spectral sequence to the filtration \ref{filtrationOX}. It can be done by considering the exact couple\begin{equation}\label{exactcouple}\cA=\bigoplus_{p,n}H^n(Y_p),  \mathcal{E}=\bigoplus_{p,n}H^n(Y_p,Y_{p-1}),\end{equation}with the usual maps $i_{p,n}:H^n(Y_{p+1})\to H^n(Y_p)$, $j_{p,n}:H^n(Y_p)\to H^{n+1}(Y_{p+1},Y_p)$ and $k_{p,n}:H^{n+1}(Y_{p+1},Y_p)\to H^{n+1}(Y_{p+1})$.

Denote by $(E_r^{p,q}(X), d_r^{p,q})$ the associated spectral sequence to the exact couple (\ref{exactcouple}), this spectral sequence converges to 
\begin{equation}\label{gradedconvergence}
E^\infty_{p,q}(X) \cong F_{p-1}(H^{p+q}(O_X))/F_p(H^{p+q}(O_X)),
\end{equation}
where $F_p(H^{p+q}(O_X))=\ker(H^{p+q}(O_X)\xrightarrow{i^*}H^{p+q}(Y_{p}))$ and $i:Y_p\to O_X$ is the inclusion, see for instance theorem 2.6 in \cite{mccleary}. We have the following result.

\begin{proposition}
The spectral sequence $(E_r^{p,q}(O_X), d_r^{p,q})$ collapses at the page two and moreover, for $r=0,...,d$
\begin{equation}\label{graded}
H^{N-r}(O_X) \cong \bigoplus_{p+q=N-r} E_2^{p,q}(O_X).
\end{equation} 
\end{proposition}

\begin{proof}
	From the exact couple we know that $$E_1^{p,q}(X)=H^{p+q}(Y_p,Y_{p-1})$$and a simple application of the long exact sequence axiom and of proposition \ref{OXTop2} gives
	$$E_1^{p,q}(X)=\begin{cases}
		\mathbb{Z}^{|F_{d-p}|}&\text{ if }q=N\\
		0&\text{ if }q\neq N
	\end{cases}$$
	On the other hand the first differential is defined as 
\begin{equation}\label{OXdiff1}
d_1^{p,q}=j_{p,n}\circ k_{p,p+q-1}:H^{p+q}(Y_p,Y_{p-1})\to H^{p+q+1}(Y_{p+1},Y_p).
\end{equation}	
	

The page $E^2$ has only one non trivial row, which is given by the cohomology of the complex:
\begin{equation}\label{E2OXpage}
0\to E_1^{0,N}\ZZ^{|F_d|}\to\cdots \to E_1^{d-1,N}\ZZ^{|F_1|}\to E_1^{d,N}\ZZ\to 0.
\end{equation}	
	It implies that $$E_\infty^{p,q}(X)=E_2^{p,q}(X).$$
	Now, for the next part of the statement we need to identify the associated graded group with the singular cohomology of $O_X$ with compact support. The proof of that fact is very similar to the identification of cellular cohomology with singular cohomology.
	
Indeed, let us consider the long exact sequence of the pair $(Y_p,Y_{p-1})$ in singular cohomology with compact support
  \begin{equation}\label{les}
  \cdots\to H^k(Y_p,Y_{p-1})\to H^k(Y_p)\to H^k(Y_{p-1})\to H^{k+1}(Y_p,Y_{p-1})\to\cdots.
  \end{equation}
Because of the fact that $Y_p/Y_{p-1}$ is homeomorphic to a finite disjoint union of $\mathbb{R}^{N-d+p}$ for $p=0,...,d$ we have that the morphisms (induced from the canonical inclusions)
	\begin{equation}
	H^k(Y_p)\to H^k(Y_{p-1})
	\end{equation}
	are isomorphisms for $k> N-d+p$ and for $k\leq N-d+p-2$, injective for $k=N-d+p-1$ and surjective for $k=N-d+p$. 
A direct computation gives that, for $r=0,..,d$,
\begin{equation}
F_{p}H^{N-r}(O_X)=H^{N-r}(O_X) \,\,\text{for}\,\, p=-1,...,d-r-1,
\end{equation}
and
\begin{equation}
F_{p}H^{N-r}(O_X)=0 \,\,\text{for}\,\, p=d-r,...,d.
\end{equation}

The statement (\ref{graded}) above follows now by (\ref{gradedconvergence}).

\end{proof}

\section{K-theory vs Conormal homology and Fredholm Perturbation characterisation}

 \subsection{Conormal homology}
 Conormal homology is introduced  (under a different name) and studied in  \cite{Bunke}. In \cite{CarLes}, a slighty different presentation of this homology is given, after the observations that it  coincides with the $E^2$ page of the spectral sequence computing $K_*(C^*(\Gamma_b(X)))$ and that it should provide easily computable obstructions to various Fredholm perturbations properties. We just briefly recall the definition of the chain complex and of the differential of conormal homology (see \cite{Bunke,CarLes} for more details). 
 
 With the same notation as above, the  chain complex $C_*(X)$  is the $\ZZ$-module where $C_p(X)$  is generated by 
 \begin{equation}
 \{ f\otimes \varepsilon \ ;\ f\in F_p \text{ and } \varepsilon \text{ is an orientation of } N_f   \}.
\end{equation}
Here   $N_f = (T_fX/T f)^*$ is the conormal bundle of $f\subset X$. Note that this bundle is always trivialisable with $e_i=d\rho_i, \ i\in I(f)$ as a prefered global basis, and oriented by $\epsilon_{I(f)}= \wedge_{i\in I(f)}e_i$ or its opposite.  
We define the differential  $\delta_* : C_*(X)\to C_{*-1}(X)$ by 
\begin{equation}\label{diffpcn1}
   \delta_p(f\otimes \epsilon) = \sum_{\substack{g\in F_{p-1},  \\ f\subset\overline{g}}} g\otimes e_{i_{(g,f)}}\lrcorner\epsilon
\end{equation}
where the index $i(f,g)$ and correspondant defining function $\rho_{i(f,g)}$ defines $f$ in $g$ and where $\lrcorner$ denotes the contraction of exterior forms. The {\it conormal homology} of $X$, denoted by $H^{\cn}_*(X)$ is defined to be the homology of $(C_*(X),\delta_*)$. Even and odd groups are called  {\it periodic conormal homology}:
 \begin{equation}
 H^{\pcn}_{0}(X)=\oplus_{p\ge 0} H^{\cn}_{2p}(X) \text{ and }H^{\pcn}_{1}(X)=\oplus_{p\ge 0} H^{\cn}_{2p+1}(X).
\end{equation}
We can consider conormal homology with rational coefficients as well.

\subsection{The cohomology of $O_X$ and the conormal homology of $X$ are isomorphic}\label{subsechomOXvscnhom}

We will construct  an explicit isomorphism
\begin{equation}
H^{ev/odd}(O_X)\stackrel{B}{\longrightarrow}H_{ev/odd}^{pcn}(X)
\end{equation}
where in the left hand side $H^{ev/odd}$ stands for singular cohomology (with compact supports) with integer coefficients. 
For this it will be enough, after the last proposition, to explicitly compute the first differentials $d_1^{p,N-d}$, (\ref{OXdiff1}) above.

We start by fixing an $\alpha\in H_1(\mathbb{R})$ with $\alpha \mapsto 1$ under the connecting map (which is an isomorphism) associated to the inclusion of $\{0\}$ in $\mathbb{R}_+$. Let now $\beta\in H^1(\mathbb{R})$ be such that $(\alpha,\beta)\mapsto 1$.

Let $f\in F_p$, there is a canonical homeomorphism $\phi_f:\mathbb{R}^N_f\to \mathbb{R}^{N-p}$ where $\mathbb{R}^{N-p}$ is the usual euclidien space. We let $\beta_f\in H^{N-p}(\mathbb{R}_f^{N})$ be the generator given by the image of $(\beta,...,\beta)\in (H^1(\mathbb{R}))^{N-p}$ by the product isomorphism  
\begin{equation}
\xymatrix{
H^1(\mathbb{R})\otimes \cdots \otimes H^1(\mathbb{R})\ar[r]^-\cong &H^{N-p}(\mathbb{R}^{N-p}),
}
\end{equation}
where in the left hand side there are exactly $N-p$ copies of $H^1(\mathbb{R})$, followed by the isomorphism in cohomology 
\begin{equation}
\xymatrix{
H^{N-p}(\mathbb{R}^{N-p})\ar[r]^-{(\phi_f)^*}_-\cong & H^{N-p}(\mathbb{R}_f^{N})
}
\end{equation}
induced by $\phi_f$.

By construction, for every $p=0,...,d $ we have a basis $(\beta_f)_{f\in F_p}$ of $H^{N-p}(Y_{d-p},Y_{d-p-1})$ via the isomorphism induced from proposition \ref{OXTop2}. We can now prove the following

\begin{proposition}\label{d1computation}
With the notations above we have that for $f\in F_{d-p}$ and $g\in F_{d-p-1}$ with $f\subset \overline{g}$ the following holds:
\begin{equation}
d_1^{p,N-d}(\beta_f)=\sigma(f,g)\cdot \beta_g,
\end{equation}
where $\sigma(f,g)=(-1)^{j-1}$ with $j$ the place of the coordinate in the multi-index $I(f)$ whose associated index $i(f,g)$ and correspondant defining function $\rho_{i(f,g)}$ defines $f$ in $g$.
\end{proposition}
\begin{proof}
By construction the differential
$$d_1^{p,N-d}:H^{N-d+p}(Y_p,Y_{p-1})\to H^{N-d+p+1}(Y_{p+1},Y_p)$$ 
is given by the connecting morphism in cohomology associated to the inclusion
$$Y_p\setminus Y_{p-1}\hookrightarrow Y_{p+1}\setminus Y_{p-1}.$$
Hence, by proposition \ref{OXTop2}, we are led to compute the connecting morphism
\begin{equation}
H^{N-d+p}(\mathbb{R}^{N}_f)\stackrel{d_{(f,g)}}{\longrightarrow}H^{N-d+p+1}(\mathbb{R}^{N}_g)
\end{equation}
associated to the canonical inclusion
$$\mathbb{R}^{N}_f\hookrightarrow \mathbb{R}^{N}_{(f,g)}.$$
Now, it is just a simple algebraic topology exercise to show that
$$d_{(f,g)}(\beta_f)=\sigma(f,g)\cdot \beta_g$$
from which we conclude the proof.
\end{proof}

From the last two propositions we obtain the following corollary.

\begin{corollary}\label{corhom(OX)}
For every $p=0,...,d $ the isomorphism $H^{N-p}(Y_{d-p},Y_{d-p-1})\to C_p^{cn}(X)$ given in a basis by 
$\beta_f\mapsto f\otimes \epsilon_{I(f)}$ induces an isomorphism
\begin{equation}
B_p:H^{N-p}(O_X)\stackrel{}{\longrightarrow}H_p^{cn}(X).
\end{equation}
In particular, since $N$ is even there are induced isomorphisms between the periodic versions
\begin{equation}
B_{ev/odd}:H^{ev/odd}(O_X)\stackrel{}{\longrightarrow}H_{ev/odd}^{cn}(X).
\end{equation}
\end{corollary}

\subsection{K-theory computation and Fredholm perturbation characterisation}

In this final section we can state the following K-theoretical computation which is our main result and a corollary of the results of the precedent sections.

\begin{corollary}\label{ThhmKvsH}
For every connected manifold with corners $X$ there are morphisms
\begin{equation}\label{thm:big-iso}
T_{ev/odd}:K_{ev/odd}(\cK_b(X))\longrightarrow  H_{ev/odd}^{cn}(X,\QQ).
\end{equation}
inducing rational isomorphisms.
Explicitly, $T_*$ is given by the composition of
\begin{enumerate}
\item The Connes-Thom isomorphism 
\begin{equation}
CT_h:K_*(\cK_b(X))\stackrel{}{\longrightarrow} K^*_{top}(O_X),
\end{equation}
\item the Chern character morphism
\begin{equation}
K^*_{top}(O_X)\stackrel{ch}{\longrightarrow}H^{ev/odd}(O_X,\QQ)
\end{equation}
which is a rational isomorphism
and
\item the isomorphism 
\begin{equation}
B_{ev/odd}\otimes Id:H^{ev/odd}(O_X)\otimes \mathbb{Q}\stackrel{}{\longrightarrow}H_{ev/odd}^{pcn}(X)\otimes \mathbb{Q},
\end{equation}
described in the last section.
\end{enumerate}
\end{corollary}

It is not sure that the isomorphism \eqref{thm:big-iso} holds true without tensorisation by rational numbers in general. However, this is obviously true when the Chern character homorphism  used in the construction is an integral isomorphism. This is the case for instance if the conormal homology groups are torsion free, as one can see using the classic arguments of Atiyah-Hirzebruch spectral sequences. Concrete situations where this freeness holds true exist: manifolds of codimension at most $3$, product of such manifolds with one factor of codimension at most $2$. 

In conclusion, we emphasize that by the last corollary and theorem \ref{AnavsFredthm1}, Fredholm property of elliptic $b$-operators is characterized by classes in the $K$-theory of $\cK_b(X)$ and that this group, up to torsion, identifies with the conormal homology of $X$, whose computation is elementary. This brings in a simplified obstruction for Fredholm property that can be expressed with the natural map:  
\begin{equation}
K^0_{top}(^bT^*X)\stackrel{Ind^X_{cn}}{\longrightarrow} H_{ev}^{pcn}(X)\otimes \mathbb{Q},
\end{equation}
which is the composition of the analytical index  (the obstruction to Fredholm property itself)  with our isomorphism $T_{ev}$. This raises further questions as: can one find formulas for the classes
$Ind^X_{cn}([\sigma_D])$ in terms of corners cycles or in more topological terms (in the spirit of the  Atiyah-Singer theorem) ? Is it possible to refine the computation at the integral level ?  Is there a cobordism invariance of the obstruction ? These questions will be investigated in future works. 


%
%
%

\bibliographystyle{plain}
\bibliography{CornersAindex} 

\end{document}